\documentclass[11pt,final]{article}
\usepackage{amsthm,amsfonts,amssymb, graphicx,hyperref,tikz,enumerate,psfrag,multirow,color,colortbl}
\usepackage{fullpage}

\usepackage[all]{xy}
\usepackage[utf8]{inputenc} 
\RequirePackage{amsthm,amsmath,amsfonts,tikz}  
\usetikzlibrary{arrows,automata}

\newtheorem{lemma}{Lemma}

\newtheorem{question}{Question}

\newtheorem{theorem}[lemma]{Theorem}
\newtheorem{corollary}[lemma]{Corollary}
\newcommand{\A}{\mathbb {A}}
\newcommand{\EE}{\mathbb {E}}
\newcommand{\PP}{\mathbb{P}}

\newcommand{\R}{\mathbb {R}}
\newcommand{\C}{\mathbb {C}}
\newcommand{\HH}{\mathbb{H}}
\newcommand{\cc}{\mathcal {C}}
\newcommand{\cH}{\mathcal {H}}
\newcommand{\cO}{\mathcal {O}}
\newcommand{\cS}{\mathcal {S}}
\newcommand{\cB}{\mathcal {B}}
\newcommand{\cI}{\mathfrak {S}}

\newcommand{\cL}{\mathcal {L}}
\newcommand{\cP}{\mathcal {P}}
\newcommand{\s}{\mathfrak{s}}
\newcommand{\e}{\mathfrak{e}}
\def\cGs{{\mathcal G}_{\star}}
\def\cGss{{\mathcal G}_{\star\star}}
\def\ER{Erd\H{o}s-R\'enyi }
\newcommand{\ii}{{\mathfrak i}}

\title{Spectral atoms of unimodular random trees}
\author{Justin Salez}
\begin{document}
\maketitle
\begin{abstract}
We use the Mass Transport Principle to analyze the local recursion governing the resolvent $(A-z)^{-1}$ of the adjacency operator of unimodular random trees. In the limit where the complex parameter $z$ approaches a given location $\lambda$ on the real axis, we show that this recursion induces a  decomposition of the tree into finite blocks whose geometry directly determines the spectral mass at $\lambda$. We then exploit this correspondence to obtain precise information on the pure-point support of the spectrum, in terms of expansion properties of the tree. In particular, we deduce that the pure-point support of the spectrum of any unimodular random tree with minimum degree $\delta\ge 3$ and maximum degree $\Delta$ is restricted to finitely many points, namely the eigenvalues of trees of size less than $\frac{\Delta-2}{\delta-2}$. More generally, we show that the restriction $\delta\ge 3$ can be weakened to $\delta\ge 2$, as long as the anchored isoperimetric constant of the tree remains bounded away from $0$. This applies in particular to any unimodular Galton-Watson tree without leaves, allowing us to settle a conjecture of Bordenave, Sen and Vir\'ag (2013).  
\end{abstract}

\section{Introduction}
\subsection{Motivations}
Unimodular networks are probability measures on countable rooted graphs satisfying a certain spatial invariance (see below). They were introduced in \cite{rec,obj,uni} to describe the geometry of sparse graphs when seen from a uniformly chosen vertex. These \emph{local weak limits} are often more convenient to work with than the finite-graph sequences that they approximate, and they have been shown to capture the asymptotic behavior of a number of important graph parameters. One emblematic example is the empirical distribution of the eigenvalues $\lambda_1\geq \ldots\geq \lambda_{|V|}$ of the  adjacency matrix:
\begin{eqnarray}
\label{df:esd}
\mu_G & := & \frac{1}{|V|}\sum_{k=1}^{|V|}\delta_{\lambda_k}.
\end{eqnarray}
This fundamental invariant encodes a considerable amount of information about the underlying graph $G=(V,E)$ (see, e.g., the monograph \cite{cds}), and a vast body of works has been devoted to the understanding of its typical behavior as $|V|\to\infty$, under various models. In the sparse regime $|E|=\cO(|V|)$, an elegant, unified answer can be given using the framework of local weak limits: 
\begin{theorem}[\cite{resolvent,pointwise,BorSpectrum}]
\label{th:esd}
If a sequence of finite graphs $(G_n)_{n\ge 1}$ admits a local weak limit $\cL$, then the associated sequence of empirical spectral distributions $(\mu_{G_n})_{n\ge 1}$ admits a weak limit $\mu_{\cL}\in\cP(\R)$. Moreover, the convergence holds in the Kolmogorov-Smirnov metric:
\begin{eqnarray*}
\sup_{\lambda\in\R}\left|\mu_{G_n}\left((-\infty,\lambda]\right)-\mu_{\cL}\left((-\infty,\lambda]\right)\right| & \xrightarrow[n\to\infty]{} & 0.
\end{eqnarray*}
\end{theorem}
The construction of $\mu_{\cL}$ relies on the spectral theorem for self-adjoint operators on Hilbert spaces, and will be recalled later. The essential message is that the asymptotic spectral analysis of large sparse graphs can, in principle, be performed directly at the level of their local weak limits. This program was initiated in \cite{resolvent,rank} and continued in \cite{Bordenave2015,BSV}, see also the related preprints \cite{2015arXiv150507412B,2016arXiv160902209R}. The recent survey \cite{BorSpectrum} contains a thorough exposition of the current state of the art, as well as a list of open problems. Not much is known about $\mu_\cL$, even in the important special case where $\cL$ is a  \emph{unimodular Galton-Watson tree}, i.e., a random rooted tree obtained by a Galton-Watson branching process where the root has a given offspring distribution $\pi=(\pi_k)_{k\geq 0}$ (with finite, non-zero  mean) and all descendants have the size-biased offspring distribution  $\widehat{\pi}=(\widehat{\pi}_k)_{k\geq 0}$ defined by
\begin{eqnarray}
\label{eq:sizebiased}
\widehat\pi_k & = & \frac{(k+1)\pi_{k+1}}{\sum_{i}i\pi_{i}}.
\end{eqnarray}
This particularly simple unimodular network -- henceforth denoted $\textsc{ugwt}(\pi)$ -- plays a distinguished role in the theory, since it is the local weak limit of large random graphs with asymptotic degree distribution $\pi$ \cite{BorOptimization}. This includes the popular \ER model with fixed average degree $c$ ($\pi=\textrm{Poisson}(c)$), or the random $r-$regular graph  ($\pi=\delta_r$).   In the latter case,  $\textsc{ugwt}(\pi)$ is just the infinite $r-$regular tree, whose spectrum is the well-known Kesten-McKay distribution \cite{mckay}:
\begin{eqnarray}
\label{mckay}
\mu_{\textsc{ugwt}(\delta_r)}(d\lambda) & = & \frac{r\sqrt{4(r-1)-\lambda^2}}{2\pi(r^2-\lambda^2)}{\bf 1}_{(-2\sqrt{r-1},2\sqrt{r-1})}(\lambda)\, d\lambda.
\end{eqnarray}
Apart from this degenerate case, an explicit description of $\mu_{\textsc{ugwt}(\pi)}$ seems out of reach at present, and our understanding remains extremely limited. For example, the following question was raised by Bordenave, Sen and Vir\'ag \cite[Question 1.8]{BSV} and reiterated in \cite[Question 4.5]{BorSpectrum}:
\begin{question}
\label{question}If $\pi$ has finite support and $\pi_1=0$, does $\mu_{\textsc{ugwt}(\pi)}$ admit only finitely many atoms ?
\end{question}

The aim of the present work is to provide a general understanding of the atomic mass $\mu_{\cL}(\{\lambda\})$ assigned to an arbitrary point $\lambda\in\R$, for any unimodular network $\cL$ that is concentrated on trees. Our main contribution is a general formula relating $\mu_\cL(\{\lambda\})$ to  the geometry of the connected components of a certain induced subgraph, see Theorem \ref{th:main}. Among other consequences, we answer Question \ref{question} in the affirmative, and prove that the conclusion actually extends to any unimodular random tree whose anchored isoperimetric constant is bounded away from $0$. This includes, in particular, all unimodular random trees with degrees in $\{\delta,\ldots,\Delta\}$, for any fixed $3\le\delta\le\Delta<\infty$. In addition, we provide an explicit list of all possible atoms.

\subsection{Unimodular networks and their spectral measures}
We only recall the necessary definitions, and refer to the excellent survey \cite{BorSpectrum} for details. 
\paragraph{Unimodular networks.} A \textit{rooted graph} $(G,o)$ is a graph $G=(V,E)$ together with a distinguished vertex $o\in V$, called the
\textit{root}. An isomorphism between two rooted graphs is a graph isomorphism which additionally maps the root to the root. We let $\cGs$ denote the set of isomorphism classes of connected, locally finite rooted graphs on a countable number of vertices. We turn $\cGs$ into a Polish space by defining the distance between $(G,o)$ and $(G',o')$ to be $1/(1+R)$ where 
\begin{eqnarray}
R & = & \sup\{r\ge 0\colon \cB_r(G,o)\textrm{ is isomorphic to } \cB_r(G',o')\}.
\end{eqnarray} Here, $\cB_r(G,o)$ is the ball of radius $r$ around $o$ in $G$, viewed as a rooted graph. \emph{Uniform rooting} is a natural procedure for turning a finite graph $G=(V,E)$ into a $\cGs-$valued random variable:   declare a uniformly chosen vertex $o\in V$ as the root, restrict the graph to its connected component, and forget the labels. 
If $(G_n)_{n\geq 1}$ is a sequence of finite graphs 
and if the sequence of laws induced by uniform rooting admits a limit $\cL$ in the usual weak sense for Borel probability measures on Polish spaces, we call $\cL$ the \textit{local weak limit} of $(G_n)_{n\geq 1}$.   Uniform rooting confers to $\cL$ a powerful invariance:  let $\cGss$ be the natural analogue of $\cGs$ for doubly-rooted graphs $(G,x,y)$; a $\cGs-$valued random variable $(G,o)$ (rather, its law $\cL$) is called  \emph{unimodular} if for any Borel function $f\colon \cGss\to [0,\infty]$, 
\begin{eqnarray}
\label{eq:mtp}
\EE\left[\sum_{x\in V(G)}f( G,o,x)\right] & = & 
\EE\left[\sum_{x\in V(G)}f(G,x,o)\right].
\end{eqnarray}
One may think of $f(G,x,o)$ as an amount of mass sent from  $x$ to $o$ in $G$: the equality (\ref{eq:mtp}) -- called the \emph{Mass Transport Principle}  -- then expresses the fact that the expected mass received and sent by the root coincide. It is easy to see that the local weak limit of any sequence of finite graphs is unimodular. Whether the converse holds is an open problem with deep implications, see \cite{uni,sofic,unitree}. 

\paragraph{Spectral measures.} Let $G=(V,E)$ be a countable, locally finite graph. Consider the Hilbert space $\cH=\ell^2_\C(V)$ and its canonical orthonormal basis $(\e_o)_{o\in V}$, where
\begin{eqnarray}
\e_o\colon x & \longmapsto & 
\left\{
\begin{array}{ll}
1 & \textrm{if }x=o\\
0 & \textrm{otherwise.}
\end{array}
\right.
\end{eqnarray}
By definition, the adjacency operator $A$ of $G$ is the linear operator on $\cH$ whose domain is the (dense) subspace of finitely-supported functions, and whose action on the above basis is given by
\begin{eqnarray}
\langle  \e_x | A\e_y\rangle & = & 
\left\{
\begin{array}{ll}
1 & \textrm{if }\{x,y\}\in E\\
0 & \textrm{otherwise.}
\end{array}
\right.
\end{eqnarray}
$A$ is symmetric, and this already ensures that $A-z$ is injective for each $z\in\C\setminus\R$. If in addition the range of $A-z$ is dense, then $A$ (or $G$ itself) is said to be (essentially) \emph{self-adjoint}. In that case, the resolvent $(A-z)^{-1}$ extends to a unique bounded linear operator on $\cH$ and for each $o\in V$, the spectral theorem for self-adjoint operators (see, e.g. \cite[Chapter VII]{reedsimon}) yields the representation
\begin{eqnarray}
\label{def:rootedmeasure}
\langle \e_o|(A-z)^{-1}\e_o \rangle & = & \int_\R\frac{1}{\lambda-z}\,\mu_{(G,o)}(d\lambda),\qquad (z\in\C\setminus\R)
\end{eqnarray}
for a unique Borel probability measure $\mu_{(G,o)}$ on $\R$ known as the spectral measure of the pair $(A,\e_o)$. This will be our main object of study. To gain some intuition, consider the case where $G$ is finite: then there is an orthonormal basis of eigenfunctions $\phi_1,\ldots,\phi_{|V|}$ of $A$ with respective eigenvalues $\lambda_1,\ldots,\lambda_{|V|}$, and we easily compute
\begin{eqnarray}
\label{df:finite}
\mu_{(G,o)} & = & \sum_{k=1}^{|V|}|\phi_k(o)|^2\delta_{\lambda_k}.
\end{eqnarray}
In words, $\mu_{(G,o)}$ is a mixture of atoms at the eigenvalues of $A$, their masses being the squared norm of the orthogonal projection of $\e_o$ onto the corresponding eigenspaces. In particular, we see that
\begin{eqnarray}
\label{locglo}
\mu_G & = & \frac{1}{|V|}\sum_{o\in V}\mu_{(G,o)}.
\end{eqnarray}
Since the right-hand side is nothing but an expectation under uniform rooting, it is natural to extend the definition of the spectral distribution (\ref{df:esd}) to any unimodular network $\cL$ by setting
\begin{eqnarray}
\label{df:uni}
\mu_{\cL}(\cdot) & := &  \EE\left[\mu_{({G},o)}(\cdot)\right]\qquad \textrm{ where }\qquad (G,o)\sim\cL. 
\end{eqnarray}
This is the limiting measure appearing in Theorem \ref{th:esd}. We emphasize that there are infinite graphs  whose adjacency operator is not self-adjoint, see, e.g., \cite{muller}. However, such pathological graphs have zero measure under any unimodular law (see \cite[Proposition 2.2]{BorSpectrum}), so that the definition (\ref{df:uni}) makes perfect sense. Regarding measurability issues, let us simply note that the map $(G,o)\mapsto \mu_{(G,o)}$ is continuous when restricted to self-adjoint elements of $\cGs$ (see e.g., \cite{PhD}[Lemma 2.2]). 
\begin{figure}
\begin{center}
\includegraphics[angle =0,width=10cm]{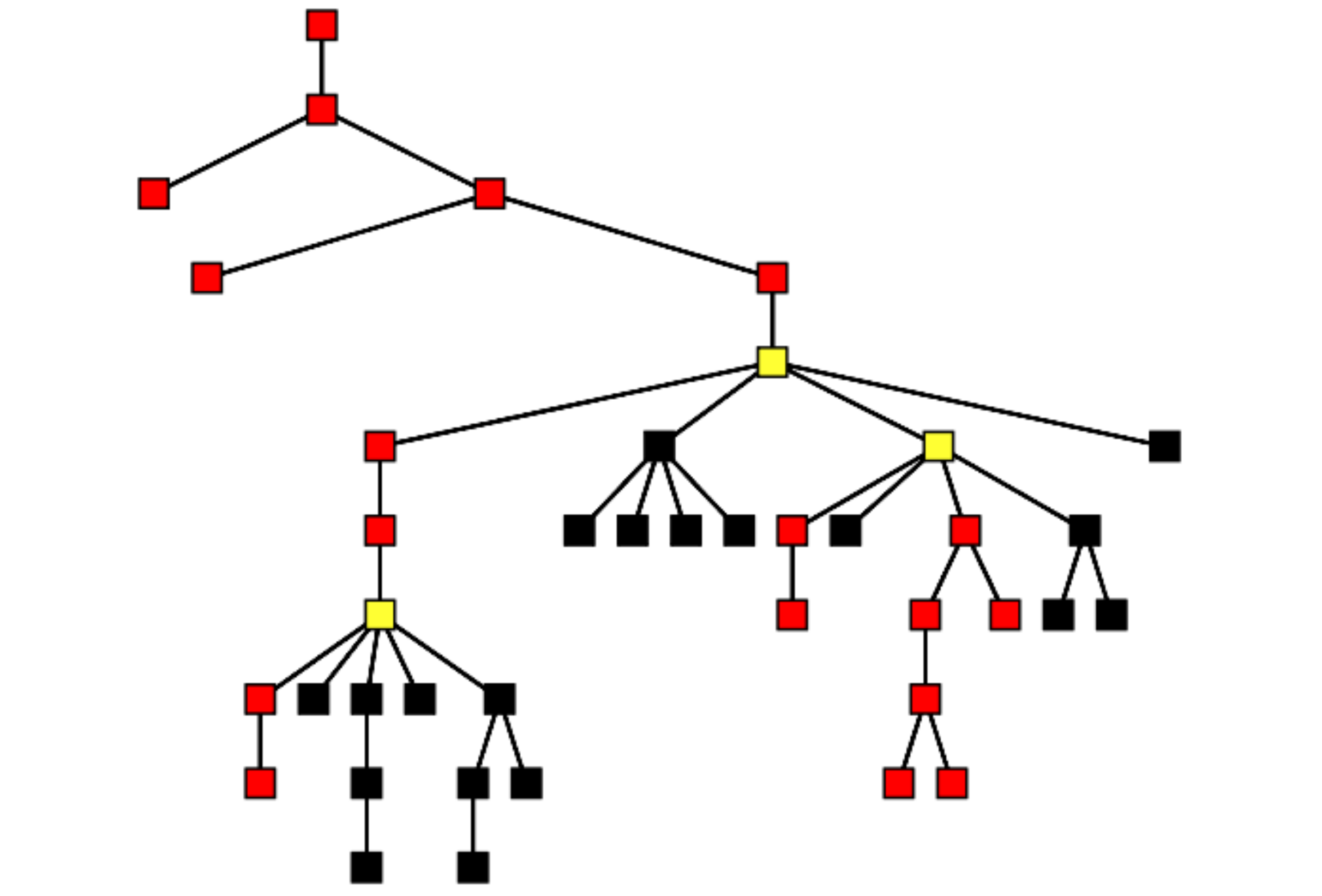}
\caption{Spectral decomposition of a $40-$vertex tree at $\lambda=1$: $\cI_1$ (in red) consists of $5$ connected components, and its boundary $\partial \cI_1$ (in yellow) has size $3$, so the multiplicity of $1$ is $5-3=2$.\label{fig:tree}}
\end{center}
\end{figure}
\subsection{Main results} 
We seek to develop a general understanding of the mass $\mu_{\cL}(\{\lambda\})=  \EE\left[\mu_{({ G},{ o})}(\{\lambda\})\right]$ assigned to an arbitrary $\lambda\in\R$. Let us define the $\lambda-$\emph{support} of a self-adjoint graph $G=(V,E)$ as
\begin{eqnarray}
\label{df:cI}
\cI_\lambda & := & \left\{o\in V\colon \mu_{(G,o)}(\{\lambda\})>0\right\}.
\end{eqnarray}
When $G$ is finite, it follows from (\ref{df:finite}) that $o\in\cI_\lambda$ if and only if there is an eigenfunction for $\lambda$ which does not vanish at $o$. Let us say that a finite graph $G=(V,E)$ is \emph{$\lambda-$prime} if $\lambda$ is an eigenvalue of the adjacency matrix of $G$ but not of $G\setminus \{o\}$, for any $o\in V$. Note that the eigenvalue $\lambda$ is then necessarily simple, with a nowhere vanishing associated eigenfunction. We use the standard notation $\deg_S(o)$ for the number of neighbors of $o\in V$ in the subset $S\subseteq V$, and $\partial S$ for the (external) \emph{boundary} of $S$, consisting of those vertices outside $S$ having at least one neighbor in $S$. We will use the short-hand $\partial o$ instead of $\partial\{o\}$ to denote the set of neighbors of $o$. 
\begin{theorem}
\label{th:main} Let $\cL$ be a unimodular network concentrated on trees, and let $\lambda\in\R$. Then almost-surely under $\cL$, the connected components induced by $\cI_\lambda$ are finite $\lambda-$prime trees, and 
\begin{eqnarray}
\label{eq:main}
\mu_\cL(\{\lambda\}) & = & \PP\left(o\in \cI_\lambda\right)-\frac 12\EE\left[\deg_{\cI_\lambda}(o){\bf 1}_{(o\in {\cI_\lambda})}\right]-\PP\left(o\in \partial {\cI_\lambda}\right).
\end{eqnarray}
\end{theorem}
This result has a simple interpretation in the case where  $\cL$ is the law induced  by uniform rooting of a finite forest: multiplying (\ref{eq:main}) by the  number of vertices, we obtain that the total multiplicity of the eigenvalue $\lambda$ in any finite forest equals the number of connected components of $\cI_\lambda$ minus $|\partial\cI_\lambda|$, as illustrated in Figure \ref{fig:tree}. In fact, the corresponding eigenspace can be completely described: any eigenfunction is zero outside $\cI_\lambda$, so its restriction to any connected component of $\cI_\lambda$ is an eigenfunction of the underlying tree. But the latter is $\lambda-$prime, so the solution is unique thereon, up to proportionality. The corresponding coefficients (one for each connected component) can be freely chosen, subject to the constraint that at each $x\in\partial\cI_\lambda$, the neighboring entries add to zero. 

Our main interest for the above formula lies in its potential to provide explicit information on the pure-point support of the spectral distribution:
\begin{eqnarray}
\label{def:support}
\Sigma_{p.p.}(\cL) & := & \left\{\lambda\in\R\colon\mu_\cL(\{\lambda\})>0\right\}.
\end{eqnarray}
For example, one of the surprising aspects of Theorem \ref{th:main} is that, even on infinite trees, the presence of a spectral atom at a given location $\lambda\in\R$ ultimately stems from the aggregation of \emph{finite} $\lambda-$prime trees. This implies in particular that $\lambda$  must belong to the (countable, dense) ring $\A$ of \emph{totally real algebraic integers}, i.e. roots of real-rooted monic polynomials with integer coefficients. 
\begin{corollary}[Algebraicity]
\label{co:algebraic}
Let $\cL$ be a unimodular network concentrated on trees. Then,
\begin{eqnarray}
\label{eq:algebraic}
\Sigma_{p.p.}(\cL)& \subseteq & \A.
\end{eqnarray}
\end{corollary}
We note that Corollary \ref{co:algebraic} also follows by combining Theorem \ref{th:esd} with a result of Elek \cite{sofic} or Benjamini, Lyons and Schramm \cite{unitree} stating that every unimodular random tree is the local weak limit of some sequence of finite graphs. Theorem \ref{th:main} provides an \emph{intrinsic} explanation for algebraicity, without resorting to finite approximation. Whether the conclusion extends to all unimodular networks -- and, in particular, to Caylay graphs of finitely generated groups -- remains an open problem. Despite its crude appearance, Corollary \ref{co:algebraic} is actually tight: the main result in \cite{eigenvalues} implies that equality holds in (\ref{eq:algebraic}) when $\cL=\textsc{ugwt}(\pi)$, 
for any degree distribution $\pi$ with full support. However, a much stronger conclusion than (\ref{eq:algebraic}) can be obtained as soon as one restricts the degrees of our unimodular random tree, as we now explain. If $o\in\cI_\lambda$, we let $\cc(\cI_\lambda,o)$ denote the connected component of $o$ in the graph induced by $\cI_\lambda$.
Using the Mass Transport Principle and the fact that $\cc(\cI_\lambda,o)$ is a finite tree, one may easily rewrite (\ref{eq:main}) into the \emph{component form}
\begin{eqnarray}
\label{eq:component}
\mu_\cL(\{\lambda\}) & = & \EE\left[\frac{{\bf 1}_{(o\in\cI_\lambda)}}{|\cc(\cI_\lambda,o)|}\left(1-\sum_{x\in\partial\cc(\cI_\lambda,o)}\frac{1}{\deg_{\cI_\lambda}(x)}\right)\right],
\end{eqnarray}
see Section \ref{sec:proof} for details. 
Now, assume that all degrees lie in $\{\delta,\ldots,\Delta\}$ for some fixed integers $\Delta\ge\delta\ge 3$. Since $|\partial S|\ge |S|(\delta-2)+2$ for any finite subset $S$ of vertices, we have the $L^\infty$ bound
\begin{eqnarray*}
\label{eq:bound}
\mu_\cL(\{\lambda\}) & \le & \frac{1}{\tau(\lambda)}\left(1-\frac{(\delta-2)\tau(\lambda)+2}{\Delta}\right)_+,
\end{eqnarray*}
where $\tau(\lambda)$ is the \emph{tree-complexity} of the totally real algebraic integer $\lambda$, defined as the minimum possible size of a tree with eigenvalue $\lambda$. This crude inequality  has a surprisingly strong consequence:
\begin{corollary}
\label{co:finite} Fix two integers $\Delta\ge\delta\ge 3$. If $\cL$ is a unimodular network concentrated on trees with degrees in $\{\delta,\ldots,\Delta\}$, then $\mu_{\cL}$ has only finitely many atoms. More precisely, 
\begin{eqnarray*}
\Sigma_{p.p.}(\cL) & \subseteq & \left\{\lambda\in\A\colon \tau(\lambda)<\frac{\Delta-2}{\delta-2}\right\}.
\end{eqnarray*}
\end{corollary}
An explicit list of the $11$ totally real algebraic integers $\lambda$ with $\tau(\lambda)\le 4$ is given in Table \ref{table:complexity} below. For example, if $\frac{\Delta-2}{\delta-2}\le 2$, then  $\Sigma_{p.p.}(\cL)\subseteq\{0\}$, while if $\frac{\Delta-2}{\delta-2}\le 3$ then $\Sigma_{p.p.}(\cL)\subseteq\{-1,0,+1\}$, and so on. In the special case of unimodular Galton-Watson trees, Corollary \ref{co:finite} answers Question \ref{question} in the affirmative, under the additional restriction that $\pi_2=0$. The constraint $\delta\ge 3$ may actually be relaxed to a control on the \emph{isoperimetric  constant}, defined for any infinite locally finite  graph $G$ by 
\begin{eqnarray}
\ii(G) & := & \inf\left\{\frac{|\partial S|}{|S|}\colon S\subseteq V(G),0<|S|<\infty\right\}.
\end{eqnarray}
Indeed, the same argument as above with  $|\partial S|\ge |S|(\delta-2)+2$ replaced by $|\partial S|\ge \ii(G)|S|$ yields
\begin{corollary} Fix $\Delta<\infty$, $\ii>0$.  If $\cL$ is a unimodular network concentrated on infinite trees with degrees at most $\Delta$ and isoperimetric constant at least $\ii$, then 
\begin{eqnarray*}
\Sigma_{p.p.}(\cL) & \subseteq & \left\{\lambda\in\A\colon \tau(\lambda)<\frac{\Delta}{\ii}\right\}.
\end{eqnarray*}
\end{corollary}
However, this is still insufficient to answer Question \ref{question} in full generality, since a unimodular Galton-Watson tree with $\pi_2>0$ contains arbitrarily long paths made of degree-2 vertices. We will overcome this by exploiting a refinement of the isoperimetric constant known as the \emph{anchored isoperimetric constant} $\ii^\star(G,o)$, defined for any (infinite) rooted graph $(G,o)\in\cGs$ by
\begin{eqnarray}
\ii^\star(G,o) & := & \lim_{n\to\infty} \inf\left\{\frac{|\partial S|}{|S|}\colon o\in S\subseteq V(G),\, G_{\restriction S} \textrm{ is connected},\, n\le |S|< \infty\right\}.
\end{eqnarray}
It is immediate to see that the choice of the root $o$ is irrelevant, and that $\ii^\star(G,o)\ge \ii(G)$. A more elaborate argument than the one used above will allow us to prove the following result.  
\begin{theorem}Fix $\Delta<\infty$, $\ii^\star>0$.
\label{th:anchored}
If $\cL$ is a unimodular network concentrated on trees with degrees in $\{2,\ldots,\Delta\}$ and {anchored} isoperimetric constant at least $\ii^\star$, then 
\begin{eqnarray*}
\Sigma_{p.p.}(\cL) & \subseteq & \left\{\lambda\in\A\colon \tau(\lambda)<\frac{3\Delta^2}{\ii^\star}\right\}.
\end{eqnarray*}
\end{theorem}
A recent result of Chen and Peres \cite[Corollary 1.3]{anchored} implies that the anchored isoperimetric constant of $\textsc{ugwt}(\pi)$ is deterministically bounded away from $0$ when $\pi_0=\pi_1=0$ and $\pi_2<1$. This finally allows us to answer Question \ref{question} in the affirmative, as promised at the beginning. 
\begin{corollary}
\label{co:conj}
If $\pi$ has finite support and $\pi_1=0$, then $\Sigma_{p.p.}({\textsc{ugwt}(\pi)})$ is finite.
\end{corollary}

\medskip

\renewcommand{\arraystretch}{2.4}
    \begin{table}[h!]
    \begin{center}
        \begin{tabular}{|>{\centering}p{3.4cm}|>{\centering}p{1.5cm}|>{\centering}p{1.5cm}|>{\centering}p{1.5cm}|>{\centering}p{1.5cm}|>{\centering}p{1.5cm}|>{\centering}p{1.5cm}|c}
        \hline
        Algebraic integer $\lambda$ & $0$ & $\pm 1$ & $\pm \sqrt{2}$ & $\pm \sqrt{3}$ & $\pm \frac{1+\sqrt{5}}{2}$ & $\pm \frac{1-\sqrt{5}}{2}$ & $\cdots$ \\
        \hline
Minimal tree &  
\centering
\begin{tikzpicture}[font=\footnotesize,scale=0.5]
\tikzstyle{solid node}=[circle,draw,inner sep=1.2,fill=black];
\node(0)[solid node]{};
\end{tikzpicture}
&
\centering
\begin{tikzpicture}[font=\footnotesize,scale=0.5]
\tikzstyle{solid node}=[circle,draw,inner sep=1.2,fill=black];
\node(0)[solid node]{}
child[grow=right]{node[solid node]{}
edge from parent node[left]{}}; 
\end{tikzpicture} & 
\begin{tikzpicture}[font=\footnotesize,scale=0.5]
\tikzstyle{solid node}=[circle,draw,inner sep=1.2,fill=black];
\node(0)[solid node]{}
child[grow=left]{node[solid node]{} edge from parent node[left]{}}
child[grow=right]{node[solid node]{} edge from parent node[left]{}}; 
\end{tikzpicture} & 
\begin{tikzpicture}[font=\footnotesize,scale=0.5]
\tikzstyle{solid node}=[circle,draw,inner sep=1.2,fill=black];
\node(0)[solid node]{}
child[grow=left]{node[solid node]{} edge from parent }
child[grow=right]{node[solid node]{} edge from parent }
child[grow=up]{node[solid node]{} edge from parent }; 
\end{tikzpicture} &
 \multicolumn{2}{c|}{
\begin{tikzpicture}[font=\footnotesize,scale=0.5]
\tikzstyle{solid node}=[circle,draw,inner sep=1.2,fill=black];
\node(0)[solid node]{}
child[grow=left]{node[solid node]{} edge from parent}
child[grow=right]{node[solid node]{} edge from parent
child[grow=right]{node[solid node]{} edge from parent }}; 
\end{tikzpicture} } & $\cdots$ \\
                \hline
                       Complexity $\tau(\lambda)$ & $1$ & $2$ & $3$ & \multicolumn{3}{c|}{$4$} & $\cdots$\\
        \hline
        \end{tabular}
        \caption{The totally real algebraic integers $\lambda$ with tree-complexity $\tau(\lambda)\le 4$.}
        \label{table:complexity}
    \end{center}
    \end{table}
We note that there is only one $0-$prime tree, namely the single vertex. Retrospectively, this simplification seems to be the essential reason behind the existence of an explicit formula for $\mu_{\textsc{ugwt}(\pi)}(\{0\})$ in terms of the generating function of $\pi$, as obtained in \cite{rank}. In contrast, there are infinitely many $\lambda-$prime trees for each $\lambda\in\A\setminus\{0\}$, leaving few hope for a general, explicit expression. The remainder of the paper is devoted to the proof of Theorem \ref{th:main} and Theorem \ref{th:anchored}. Our starting point is a well-known local recursion governing the resolvent $(A-z)^{-1}$ of self-adjoint trees. 

\section{Spectral decomposition of self-adjoint trees}
In this section, we investigate the structure of the $\lambda-$support of an arbitrary self-adjoint  tree. 
\subsection{The cavity equations}
As many graph-theoretical quantities, spectral measures 
admit a recursive structure when evaluated on trees. This recursion is better expressed in terms of the \emph{Stieltjes transform} $\s$ of a finite Borel measure $\mu$ on $\R$, defined on the upper half-plane $\HH=\{z\in\C\colon\Im(z)>0\}$ by
\begin{eqnarray}
\s(z) & := & \int_\R\frac{1}{\lambda-z}\,\mu(d\lambda).
\end{eqnarray}
Note that $\s$ is an analytic function from $\HH$ to $\HH$ with the property that
\begin{eqnarray}
\label{eq:unifbound}
\sup\left\{|\s(z)|\times \Im(z)\colon{z\in\HH}\right\} & < & \infty.
\end{eqnarray}
Conversely, it is classical that any analytic function $\s\colon\HH\to\HH$ satisfying (\ref{eq:unifbound}) is the Stieltjes transform of a unique finite Borel measure $\mu$ on $\R$. Its total mass is then given by the supremum in (\ref{eq:unifbound}). Now, if $\mu$ is a finite measure and $\s$ its Stieltjes transform, the equation
\begin{eqnarray}
\label{df:Gamma}
\widehat{\s}(z) & := & \frac{-1}{z+\s(z)}
\end{eqnarray}
defines an analytic function $\widehat{s}\colon\HH\to\HH$ satisfying (\ref{eq:unifbound}) with the supremum being $1$. Therefore, it is the Stieltjes transform of a unique probability measure on $\R$, which we will denote by $\widehat{\mu}$. The transformation $\mu\mapsto\widehat{\mu}$ will play a crucial role, for the following reason. 

Let $T=(V,E)$ be a self-adjoint tree. Deleting a vertex $o$ splits $T$ into $\deg(o)$ disjoint subtrees $(T_{x\to o}\colon x\in\partial o)$. 
 At the operator-theoretic level, this translates into the orthogonal decomposition
\begin{eqnarray}
A_{T\setminus o} & = & \bigoplus_{x\in\partial o}A_{T_{x\to o}}
\end{eqnarray}
From this identity and the fact that $A_T$ is self-adjoint, it easily follows (see, e.g., \cite{PhD}) that the $(A_{T_{x\to o}}\colon x\in\partial o)$ are themselves self-adjoint, and that the spectral measure $\mu_{(T,o)}$ is related to the spectral measures $\left(\mu_{(T_{x\to o},x)}\colon x\in\partial o\right)$ via the formula
\begin{eqnarray}
\label{decide}
\mu_{(T,o)} & = & \widehat{\left\{\sum_{x\in\partial o}\mu_{(T_{x\to o},x)}\right\}}.
\end{eqnarray} 
Similarly, for any $y\in\partial o$, the above formula applied to $T_{o\to y}$ instead of $T$ yields
\begin{eqnarray}
\label{update}
\mu_{{(T_{o\to y},o)}} & = & \widehat{\left\{\sum_{x\in\partial o\setminus \{y\}}\mu_{(T_{x\to o},x)}\right\}}.
\end{eqnarray}  
We will refer to the local identities (\ref{decide}) and (\ref{update}) as the \emph{cavity equations} at $o$ and  $(o,y)$, respectively. This recursive structure is not new (see, e.g., \cite{resolvent,rank}). Our main contribution here is to investigate the way in which  it affects the underlying measures ``locally'', at a given spectral location $\lambda\in\R$.  
\subsection{Impact at a given spectral location $\lambda\in\R$} Let us now fix a location $\lambda\in\R$ and a finite measure $\mu$, and focus on the "local" statistics
\begin{eqnarray}
\label{df:alpha}
\alpha \ := \ \mu(\{\lambda\}) & \textrm{ and } & \beta \ := \ \int_\R\frac{1}{(\xi-\lambda)^2}\, \mu(d\xi).
\end{eqnarray} 
Note that $\alpha=0$ or $\beta=\infty$.  
When $\beta<\infty$, we may safely consider the additional quantity 
\begin{eqnarray}
\label{df:gamma}
\gamma & := & \int_\R\frac{1}{\xi-\lambda}\, \mu(d\xi).
\end{eqnarray}
Perhaps surprisingly, the triple $(\alpha,\beta,\gamma)$ happens to evolve autonomously under the action of $(\ref{df:Gamma})$, in the sense that the triple $(\widehat{\alpha},\widehat{\beta},\widehat{\gamma})$ corresponding to the measure $\widehat{\mu}$ is completely determined by the triple $(\alpha,\beta,\gamma)$ only. The precise  functional relation is summarized in the following Table.
\medskip
\renewcommand{\arraystretch}{2.3}
    \begin{table}[h!]
    \begin{center}
        \begin{tabular}{|>{\centering}p{1.5cm}|>{\centering}p{1.5cm}|>{\centering}p{1.5cm}||c|c|c|}
        \cline{1-6}
         \multicolumn{3}{|c||}{\textbf{Possible cases}} & 
	     \textbf{Value of $\boldsymbol{\widehat{\alpha}}$} & \textbf{ Value of $\boldsymbol{\widehat{\beta}}$ }  & \textbf{Value of $\boldsymbol{\widehat{\gamma}}$}  \\
            \cline{1-6}
             \multicolumn{3}{|c||}{$\alpha>0$} & $0$ & $\displaystyle{\frac 1 \alpha}$ & $0$ \\
            \cline{1-6}
             \multirow{3}{*}{$\alpha=0$} 
                         & \multirow{2}{*}{$\beta<\infty$} & 
                         $\gamma = -\lambda$ & $\displaystyle{\frac 1 {1+\beta}}$  & $\infty$ & undef. \\

                          \cline{3-6}
            & & $\gamma\neq-\lambda$ & $0$ & $\displaystyle{\frac {1+\beta}{(\lambda+\gamma)^2}}$ & $\displaystyle{\frac {-1}{\lambda+\gamma}}$ \\ 
                        \cline{2-6}
& \multicolumn{2}{c||}{$\beta=\infty$}  & $0$ & $\infty$ & undef.  \\
            \hline
        \end{tabular}
        \caption{The values of $\widehat{\alpha},\widehat{\beta},\widehat{\gamma}$  as  functions of $\alpha,\beta,\gamma$. }
        \label{table:Gamma}
    \end{center}
    \end{table}
    
\begin{proof}[Proof of the relations in Table \ref{table:Gamma}]The triple $(\alpha,\beta,\gamma)$ is directly related to the behavior of the Stieltjes transform $\s$ of $\mu$ near the spectral location $\lambda$. Indeed, it follows from the definitions that
\begin{eqnarray}
\label{eq:alpha}
\s(\lambda+i\varepsilon) & = & \frac{i \alpha}{\varepsilon} + o\left(\frac 1\varepsilon\right),
\end{eqnarray}
as $\varepsilon\to 0^+$ (dominated convergence), and also (monotone convergence) that
\begin{eqnarray}
\label{eq:beta}
\beta & = & \displaystyle{\lim_{\varepsilon\to 0^+}\uparrow\frac{\Im\s(\lambda+i\varepsilon)}\varepsilon}.
\end{eqnarray}
Moreover, whenever $\beta<\infty$, we actually have the lower-order expansion
\begin{eqnarray}
\label{eq:gamma}
\s(\lambda+i\varepsilon) & = & \gamma+i\beta\varepsilon+o(\varepsilon). 
\end{eqnarray}
Injecting these into (\ref{df:Gamma}) immediately gives access to the behavior of $\widehat{\s}(\lambda+i\varepsilon)$ as $\varepsilon\to 0^+$. From the latter, we may in turn deduce the values of $\widehat{\alpha},\widehat{\beta},\widehat{\gamma}$, using (\ref{eq:alpha})-(\ref{eq:beta})-(\ref{eq:gamma}) with $\alpha,\beta,\gamma,\s$ replaced by  $\widehat{\alpha},\widehat{\beta},\widehat{\gamma},\widehat{\s}$, respectively. There are four cases, corresponding to the four rows in Table \ref{table:Gamma}:
\begin{itemize} 
\item If $\alpha>0$ then $\widehat{\s}(\lambda+i\varepsilon)\sim\frac{i\varepsilon}{\alpha}$, implying in turn that  $\widehat{\beta}=\frac1\alpha$ and $\widehat{\gamma}=0$. 
\item If $\beta<\infty$, $\gamma=-\lambda_0$ then 
$
\widehat{\s}(\lambda+i\varepsilon) \sim \frac{i}{(1+\beta)\varepsilon}$, implying  $\widehat{\alpha}=\frac{1}{1+\beta}$. 
\item If $\beta<\infty$, $\gamma\neq -\lambda_0$ then $
\widehat{\s}(\lambda+i\varepsilon) = \frac{-1}{\lambda+\gamma}+\frac{i\varepsilon(1+\beta)}{(\lambda+\gamma)^2}+o(\varepsilon)$, implying  $\widehat{\beta}=\frac{1+\beta}{(\lambda+\gamma)^2}$, $\widehat{\gamma}=\frac{-1}{\lambda+\gamma}$. 
\item If $\alpha=0$, $\beta=\infty$ then $|\widehat{\s}(\lambda+i\varepsilon)|\gg \varepsilon$ and $\frac{\Im\widehat{\s}(\lambda+i\varepsilon)}{|\widehat{\s}(\lambda+i\varepsilon)|^2}\gg\varepsilon$. Contradictorily, assuming $\widehat{\beta}<\infty,\widehat{\gamma}= 0$ would yield $|\widehat{\s}(\lambda+i\varepsilon)|\sim \widehat{\beta}\varepsilon$, assuming $\widehat{\beta}<\infty,\widehat{\gamma}\neq 0$ would yield $\frac{\Im\widehat{\s}(\lambda+i\varepsilon)}{|\widehat{\s}(\lambda+i\varepsilon)|^2}\sim \frac{\widehat{\beta}\varepsilon}{\widehat{\gamma}^2}$, and assuming $\widehat{\alpha}>0$ would yield $\frac{\Im\widehat{\s}(\lambda+i\varepsilon)}{|\widehat{\s}(\lambda+i\varepsilon)|^2}\sim\frac{\varepsilon}{\widehat{\alpha}}$. Thus, we must have $\widehat{\alpha}=0$ and $\widehat{\beta}=\infty$.
\end{itemize}This concludes the proof.
\end{proof}
\subsection{Structure of the $\lambda-$support} 
Let us now fix a self-adjoint tree $T=(V,E)$ and apply the general relations given in Table \ref{table:Gamma} to the cavity equations (\ref{decide})-(\ref{update}). For $o\in V$, we let $(\alpha_o,\beta_o,\gamma_o)$ denote the local statistics of the measure $\mu_{(T,o)}$ at  $\lambda$, as defined in (\ref{df:alpha})-(\ref{df:gamma}). We define  $(\alpha_{x\to o},\beta_{x\to o},\gamma_{x\to o})$ similarly, with $\mu_{(T_{x\to o},x)}$ instead.     
\begin{lemma}[Reciprocity relations]
\label{lm:Gamma}
Consider a vertex $o\in\cI_\lambda$. Then  
\begin{eqnarray}
\label{eqz}
\sum_{x\in\partial o}\beta_{x\to o} & < & \infty\\
\label{eqb}
\sum_{x\in\partial o}\gamma_{x\to o} & = & -\lambda\\
\label{deg}
\sum_{x\in\partial o}\alpha_o\beta_{x\to o} & = & 1-\alpha_o. 
\end{eqnarray}
Moreover, for each $y\in\partial o\cap \cI_\lambda$, 
\begin{eqnarray}
\label{prod}
\gamma_{y\to o} & = &  \frac{1}{\gamma_{o\to y}}\\
\label{sum}
\alpha_o\beta_{y\to o} & = & 1-\alpha_y\beta_{o\to y}.
\end{eqnarray}
On the other hand, for each $y\in\partial o \setminus \cI_\lambda$, 
 \begin{eqnarray}
\label{prodz}
\gamma_{y\to o} & = & 0\\
\label{sumz}
\alpha_o\beta_{y\to o} & = & \frac{\alpha_{o\to y}}{\sum_{x\in\partial y\cap\cI_\lambda}\alpha_{x\to y}}.
\end{eqnarray}
\end{lemma}
\begin{proof}
Consider the cavity equation at $o$: since only the second row in Table \ref{table:Gamma} yields $\widehat{\alpha}>0$, the assumption $o\in \cI_\lambda$ is actually equivalent to the conditions (\ref{eqz})-(\ref{eqb}), and we then have the formula
\begin{eqnarray}
\label{eqa}\alpha_o & = & \frac{1}{1+\sum_{x\in\partial o}\beta_{x\to o}}.
\end{eqnarray}
The identity (\ref{deg}) follows immediately. Now, let $y\in\partial o$. There are two possible scenarii:

\paragraph{First case: $\gamma_{y\to o}\ne 0$.} Then (\ref{eqb}) fails to hold if $\partial o$ is replaced by $\partial o\setminus\{y\}$. Consequently, the third row in Table \ref{table:Gamma} applies to the cavity equation at $(o,y)$ and yields
\begin{eqnarray}
\label{eqc}
\beta_{o\to y} & = & \frac{1+\sum_{x\in\partial o\setminus\{y\}}\beta_{x\to o}}{\left(\lambda+\sum_{x\in\partial o\setminus\{y\}}\gamma_{x\to o}\right)^2}\\
\label{eqd}
\gamma_{o\to y} & = & \frac{-1}{\lambda+\sum_{x\in\partial o\setminus\{y\}}\gamma_{x\to o}}.
\end{eqnarray}
Note that by (\ref{eqb}), the denominator on the right-hand side of (\ref{eqd}) equals $-\gamma_{y\to o}$, and (\ref{prod}) follows. Now, let us take a look at the cavity equation at $(y,o)$: since only the third row in  Table \ref{table:Gamma} yields $\widehat{\gamma}\neq 0$, the fact that $ \gamma_{y\to o}\neq 0$ ensures that 
\begin{eqnarray}
\label{eqdd}
\gamma_{y\to o} & = & \frac{-1}{\lambda+\sum_{x\in\partial y\setminus\{o\}}\gamma_{x\to y}}.
\end{eqnarray}
Replacing the left-hand side by $1/\gamma_{o\to y}$ and rearranging, we see that 
\begin{eqnarray}
\sum_{x\in\partial y}\gamma_{x\to y} & = & -\lambda.
\end{eqnarray}
Thus, the second row in  Table \ref{table:Gamma} applies to the cavity equation at $y$, and we conclude that $y\in \cI_\lambda$.  Finally, we may use (\ref{eqd}) and  (\ref{prod}) to rewrite (\ref{eqc}) as 
\begin{eqnarray*}
{1+\sum_{x\in\partial o\setminus\{y\}}\beta_{x\to o}} & = & \frac{\beta_{o\to y}}{\left(\gamma_{o\to y}\right)^2}\\ & = & \frac{\gamma_{y\to o}\beta_{o\to y}}{\gamma_{o\to y}},
\end{eqnarray*}
which we may then insert into  (\ref{eqa}) to arrive at
\begin{eqnarray*}
\alpha_o & = & \frac{\gamma_{o\to y}}{\gamma_{o\to y}\beta_{y\to o}+\gamma_{y\to o}\beta_{o\to y}}.
\end{eqnarray*}
By symmetry, the same formula holds with $o$ and $y$ interchanged, and (\ref{sum}) follows. 
\paragraph{Second case: $\gamma_{y\to o}=0$.} Then (\ref{eqz})-(\ref{eqb}) continue to hold with $\partial o$ replaced by $\partial o\setminus\{y\}$. Thus, the second row in Table \ref{table:Gamma} applies to the cavity equation at $(o,y)$ and yields
\begin{eqnarray}
\label{use1}
\alpha_{o\to y} & = & \frac 1{1+\sum_{x\in\partial o\setminus\{y\}}\beta_{x\to o}}.
\end{eqnarray}
In particular, $\alpha_{o\to y}>0$. Thus, the first row in Table \ref{table:Gamma} applies to the cavity equation at $y$, and we conclude that $y\notin\cI_\lambda$. Finally, consider the cavity equation at $(y,o)$:  since only the first row in Table \ref{table:Gamma} yields $\widehat{\gamma}=0$, the assumption $\gamma_{y\to o}=0$ ensures that  
\begin{eqnarray}
\label{use2}
\beta_{y\to o} & = & \frac{1}{\sum_{x\in\partial y\setminus \{o\}}\alpha_{x\to y}}.
\end{eqnarray} 
We may now use (\ref{eqa}), (\ref{use1}) and (\ref{use2}) successively to write
\begin{eqnarray*}
\alpha_o\beta_{y\to o} & = & \frac{\beta_{y\to o}}{\beta_{y\to o}+\frac{1}{\alpha_{o\to y}}}\\ & = & \frac{\alpha_{o\to y}}{\alpha_{o\to y}+\frac{1}{\beta_{y\to o}}}\\
& = & \frac{\alpha_{o\to y}}{\sum_{x\in\partial y}\alpha_{x\to y}}.
\end{eqnarray*}
To obtain (\ref{sumz}), it remains to argue that only those $x\in\cI_\lambda$ contribute to the denominator. To see this, consider the cavity equation at $(y,x)$ for an arbitrary $x\in\partial y\setminus \{o\}$: since $\alpha_{o\to y}>0$, the first row in Table \ref{table:Gamma}  guarantees that $\beta_{y\to x}<\infty$ and $\gamma_{y\to x}=0$. In other words, $y$ does not contribute to the criterion for whether $\alpha_x>0$, and the latter becomes equivalent to the condition $\alpha_{x\to y}>0$. 
\end{proof}
\begin{lemma}\label{lm:boundary}
Any vertex in the boundary of $\cI_\lambda$ actually has at least \underline{two} neighbors in $\cI_\lambda$, i.e.
\begin{eqnarray*}
o\in\partial\cI_\lambda & \Longrightarrow & \deg_{\cI_\lambda}(o)\ge 2.
\end{eqnarray*}
\end{lemma}
\begin{proof}
If $o\in\partial\cI_\lambda$ then by definition, there is some neighbor $x$ of $o$ that lie inside $\cI_\lambda$. By (\ref{prodz}), we know that $\gamma_{o\to x}=0$. Since only the first row in Table \ref{table:Gamma} yields $\widehat{\gamma}=0$, we deduce that 
\begin{eqnarray*}
\sum_{y\in\partial o\setminus x}\alpha_{y\to o} & > & 0. 
\end{eqnarray*}
In other words, $o$ admits another neighbor $y$, distinct from $x$, such that $\alpha_{y\to o}>0$. This actually implies that $y\in\cI_\lambda$, as shown at the end of the above proof. 
\end{proof}
\section{Proofs of the main results}
\label{sec:proof}
We may finally specialize the above deterministic identities to unimodular random trees, and exploit the Mass Transport Principle to establish Theorems \ref{th:main} and \ref{th:anchored}.
\subsection{Proof of Theorem  \ref{th:main}}
\paragraph{The main formula (\ref{eq:main}).}  
For a self-adjoint tree $T=(V,E)$ and a vertex $o\in V$, we have by (\ref{sum})   
\begin{eqnarray*}
\deg_{\cI_\lambda}(o){\bf 1}_{(o\in \cI_\lambda)} & = & 
 \sum_{x\in\partial o} (\alpha_o\beta_{x\to o}+\alpha_x\beta_{o\to x}){\bf 1}_{(x\in \cI_\lambda)}{\bf 1}_{(o\in \cI_\lambda)}.
\end{eqnarray*}
Taking expectation at the root of our unimodular random tree, we obtain
\begin{eqnarray*}
\EE\left[\deg_{\cI_\lambda}(o){\bf 1}_{(o\in \cI_\lambda)}\right] & = & \EE\left[\sum_{x\in\partial o}\alpha_o\beta_{x\to o}{\bf 1}_{(x\in \cI_\lambda)}{\bf 1}_{(o\in \cI_\lambda)}\right]+\EE\left[\sum_{x\in\partial o}\alpha_x\beta_{o\to x}{\bf 1}_{(x\in \cI_\lambda)}{\bf 1}_{(o\in \cI_\lambda)}\right].
\end{eqnarray*}
By the Mass Transport Principle, the two terms on the right-hand side are equal and hence
\begin{eqnarray*}
\EE\left[\sum_{x\in\partial o}\alpha_o\beta_{x\to o}{\bf 1}_{(o\in \cI_\lambda)}{\bf 1}_{(x\in \cI_\lambda)}\right] & = & \frac{1}{2}\EE\left[\deg_{\cI_\lambda}(o){\bf 1}_{(o\in \cI_\lambda)}\right].
\end{eqnarray*}
 On the other hand, using (\ref{sumz}), we have 
\begin{eqnarray*}
\EE\left[\sum_{x\in\partial o}\alpha_o\beta_{x\to o}{\bf 1}_{(o\in\cI_{\lambda})}{\bf 1}_{(x\in\partial\cI_\lambda)}\right] & = & \EE\left[\sum_{x\in\partial o}\frac{\alpha_{o\to x}}{\sum_{y\in \partial x\cap\cI_\lambda}\alpha_{y\to x}}{\bf 1}_{(o\in\cI_{\lambda})}{\bf 1}_{(x\in\partial \cI_\lambda)}\right] \\
& = & \EE\left[\sum_{x\in\partial o}\frac{\alpha_{x\to o}}{\sum_{y\in \partial o\cap\cI_\lambda}\alpha_{y\to o}}{\bf 1}_{(x\in\cI_{\lambda})}{\bf 1}_{(o\in\partial \cI_\lambda)}\right]\\
& = & \PP\left(o\in\partial \cI_\lambda\right),
\end{eqnarray*}
by the Mass Transport Principle again.
Adding-up those two identities, we obtain 
\begin{eqnarray*}
\PP\left(o\in\partial \cI_\lambda\right)+\frac{1}{2}\EE\left[\deg_{\cI_\lambda}(o){\bf 1}_{(o\in\cI_\lambda)}
\right] & = & \EE\left[
\sum_{x\in\partial o}\alpha_o\beta_{x\to o}{\bf 1}_{(o\in\cI_{\lambda})}\right]\\
& = & \EE\left[(1-\alpha_o){\bf 1}_{(o\in\cI_{\lambda})}\right],
\end{eqnarray*}
thanks to (\ref{deg}). Recalling that $\EE[\alpha_o]=\mu_{\cL}(\{\lambda\})$ concludes the proof of the identity (\ref{eq:main}). 
\paragraph{Finiteness of the connected components.}
Let $\cI_\lambda^\star$ denote the restriction of $\cI_\lambda$ to its infinite connected components. In order to show that $\cI_\lambda^\star=\emptyset$ almost-surely, it is enough to prove $\PP\left(o\in \cI_\lambda^\star\right)=0$, since ``everything shows up at the root'' (see \cite{uni}[Lemma 2.3]). Now, the same argument as in the above proof shows that 
\begin{eqnarray*}
\EE\left[\deg_{\cI_\lambda^\star}(o){\bf 1}_{(o\in \cI_\lambda^\star)}\right] & = & 2
\EE\left[\sum_{x\in\partial o}\alpha_o\beta_{x\to o}{\bf 1}_{(o\in \cI_\lambda^\star)}{\bf 1}_{(x\in \cI_\lambda^\star)}\right]\\
 & \le & 2\EE\left[\sum_{x\in\partial o}\alpha_o\beta_{x\to o}{\bf 1}_{(o\in \cI_\lambda^\star)}\right]\\ & \le & 2\EE\left[\left(1-\alpha_o\right){\bf 1}_{(o\in \cI_\lambda^\star)}\right],
\end{eqnarray*}
where in the last line, we have used (\ref{deg}). Consequently, if $\PP(o\in \cI_\lambda^\star)>0$, then
\begin{eqnarray*}
\EE\left[\deg_{\cI_\lambda^\star}(o)\,\big|\, o\in \cI_\lambda^\star\right] & < & 2.
\end{eqnarray*}
By a classical result of Aldous and Lyons \cite{uni}[Theorem 6.1], this strict inequality contradicts the fact that all components of  $\cI_\lambda^\star$ are infinite with probability $1$. Thus, $\PP(o\in \cI_\lambda^\star)=0$, as desired.

\paragraph{The alternative formulation (\ref{eq:component}).} Let us now rewrite the main formula (\ref{eq:main}) into the component form (\ref{eq:component}) using the fact that the components of $\cI_\lambda$ are finite trees. If $x,y$ are neighbors and both lie in $\cI_\lambda$, define the mass sent from $x$ to $y$ to be $|K'|/|K|$, where $K$ denotes the connected component of $\cI_\lambda$ containing $x$ and $y$, and $K'$  the restriction of $K$  to those vertices that are closer to $x$ than to $y$. Otherwise, the mass sent is zero. The Mass Transport Principle then reads
\begin{eqnarray*}
\EE\left[\left(1-\frac{1}{|\cc(\cI_\lambda,o)|}\right){\bf 1}_{(o\in\cI_\lambda)}\right] & = & \EE\left[\left(\deg_{\cI_\lambda}(o)+\frac{1}{|\cc(\cI_\lambda,o)|}\right){\bf 1}_{(o\in\cI_\lambda)}\right].
\end{eqnarray*}
Rearranging, we see that 
\begin{eqnarray}
\label{eq:deg}
\PP\left(o\in\cI_\lambda\right)-\frac 12\EE\left[\deg_{\cI_\lambda}(o){\bf 1}_{(o\in\cI_\lambda)}\right] & = &  
\EE\left[\frac{{\bf 1}_{(o\in\cI_\lambda)}}{|\cc(\cI_\lambda,o)|}\right].
\end{eqnarray}
Now, let us define another mass transport: the mass sent from any $x$ to any $y$ is $1/(|\cc(\cI_\lambda,y)|\deg_{\cI_\lambda}(x))$ if $y\in\cI_\lambda$ and $x\in\partial \cc(\cI_\lambda,y)$, and zero otherwise. Then the Mass Transport Principle reads
\begin{eqnarray*}
\EE\left[\frac{{\bf 1}_{(o\in\cI_\lambda)}}{|\cc(\cI_\lambda,o)|}\sum_{x\in\partial\cc(\cI_\lambda,o)}\frac{1}{\deg_{\cI_\lambda}(x)}\right] & = & \PP\left(o\in\partial\cI_\lambda\right).
\end{eqnarray*}
Combining those two identities shows that the main formula (\ref{eq:main}) is indeed equivalent to (\ref{eq:component}).
\paragraph{$\lambda-$primality of the connected components.}  Take a self-adjoint tree  $T=(V,E)$, and consider any finite connected component $S$ of the subgraph induced by $\cI_\lambda$. By (\ref{prod}), 
\begin{eqnarray*}
\gamma_{y\to x} & = & \frac{1}{\gamma_{x\to y}},
\end{eqnarray*}
for any internal edge $\{x,y\}$ in $S$. This implies that we may write 
\begin{eqnarray*}
\gamma_{y\to x} & = & \frac{-\phi(y)}{\phi(x)},
\end{eqnarray*}
for some function $\phi\colon S\to\R\setminus\{0\}$ (unique, up to proportionality). On the other hand, the identities (\ref{prod}) and (\ref{prodz}) together ensure that for any $x\in S$,
\begin{eqnarray*}
-\sum_{y\in \partial x\cap S} \gamma_{y\to x} & = & \lambda.
\end{eqnarray*}
Those two equations together show that $\phi$ is a (non-vanishing) eigenfunction of $S$ associated with the eigenvalue $\lambda$. This is enough to conclude, thanks to the following characterization. 

\begin{lemma}
Let $\lambda\in\R$ and let $T$ be a finite tree. Then the following are equivalent:
\begin{enumerate}[(i)]
\item $T$ admits a nowhere-vanishing eigenfunction for the eigenvalue $\lambda$. 
\item $T$ is $\lambda-$prime, i.e., $\lambda$ is an eigenvalue of $T$ but not of $T\setminus o$, for any $o\in V$.  
\end{enumerate}
\end{lemma}
\begin{proof}
The implication $(ii)\Longrightarrow (i)$ is easy and actually holds for arbitrary graphs, as already noted below the definition of  $\lambda-$primality. We focus on the converse. Fix a finite tree $T=(V,E)$ and a function $\phi\colon V\to\R\setminus\{0\}$ satisfying, at every $x\in V$, 
\begin{eqnarray}
\label{eigen}
\sum_{y\in\partial x}\phi(y) & = & \lambda\phi(x)
\end{eqnarray}
Now, consider an oriented edge $(o,o')$ and a function $\varphi\colon V\to\R$ and suppose that $\varphi$ satisfies (\ref{eigen}) at every vertex $x$ of $T_{o\to o'}$, except maybe at $x=o$. We claim that the restriction $\varphi_{\restriction T_{o\to o'}}$ is then proportional to $\phi_{\restriction T_{o\to o'}}$. Let us prove this by induction on the height of the tree $T_{o\to o'}$. If it is reduced to its root $o$, then the claim is trivial. Now, inductively, assume that the claim holds for the trees $T_{x_1\to o},\ldots,T_{x_d\to o}$, where $x_1,\ldots,x_d$ are the neighbors of $o$ other than $o'$. This means that there exists proportionality constants $\kappa_1,\ldots,\kappa_d$ such that $\varphi_{\restriction T_{x_i\to o}} = \kappa_i\phi_{\restriction T_{x_i\to o}}$ for each $1\le i\le d$. 
Comparing the eigenvalue equation for $\varphi$ and for $\kappa_i\phi$ at $x=x_i$, we immediately deduce that
\begin{eqnarray}
\varphi(o) & = & \kappa_i\phi(o).
\end{eqnarray}
Since $\phi(o)\ne 0$, we conclude that $\kappa_1=\cdots=\kappa_d=\frac{\varphi(o)}{\phi(o)}$, and this  precisely means that $\varphi_{\restriction T_{o\to o'}}$ is proportional to $\phi_{\restriction T_{o\to o'}}$, as desired. We now have all we need to prove (ii): fix $o\in V$ and suppose that $T\setminus o$ admits an eigenfunction $\varphi$ associated with the eigenvalue $\lambda$. Extend it to a function on $V$ by setting $\varphi(o)=0$. Then $\varphi$ satisfies (\ref{eigen}) at every $x\in V\setminus o$, and the above argument shows that it must be proportional to $\phi$ everywhere on $V$. Since $\phi(o)\ne 0$, this forces $\varphi\equiv 0$, as desired. 
\end{proof}
\subsection{Proof of Theorem \ref{th:anchored}}
We split the proof into two parts: Lemma \ref{lm:thin} converts the presence of spectral atoms with high tree-complexity into the existence of certain \emph{thin} sets in the tree. In turn, Lemma \ref{lm:anchored} shows that the presence of such thin sets forces  the anchored isoperimetric constant of the tree to be small. 
\begin{lemma}\label{lm:thin}Let $\cL$ be a unimodular network concentrated on trees with all degrees in $\{2,\ldots,\Delta\}$, and let $\lambda\in\Sigma_{p.p.}(\cL)$. Set $\varepsilon=\frac{2(\Delta-2)}{\tau(\lambda)}$. Then $\cS := \cI_\lambda\cup\partial\cI_\lambda$ is \emph{$\varepsilon-$thin} in the following sense:
\begin{enumerate}[(i)]
\item $\cS$ has positive density, i.e. $\PP\left(o\in\cS\right)>0$.
\item Vertices in $\cS$ have internal degree at least two, i.e. $\PP\left(\deg_\cS(o)\ge 2\,|\,o\in\cS\right)=1$.
\item Vertices in $\cS$ have total degree (internal+external) close to two, i.e.  $\EE\left[\deg(o)-2\,|\,o\in\cS\right]\le \varepsilon.$
\end{enumerate}
\end{lemma}
\begin{proof}
The condition (i) is clear, since $\lambda\in\Sigma_{p.p}(\cL)$. The condition (ii) follows from Lemma \ref{lm:boundary} when $o\in\partial\cI_\lambda$, and from the assumption on $\cL$ when $o\in\cI_\lambda$. We now focus on (iii). 
If each $x\in \cI_\lambda$ sends mass $1$ to each $y\in\partial x\setminus\cI_\lambda$, then the Mass Transport Principle reads
\begin{eqnarray*}
\EE\left[\deg_{\partial \cI_\lambda}(o){\bf 1}_{(o\in\cI_\lambda)}\right] 
& = & \EE\left[\deg_{\cI_\lambda}(o){\bf 1}_{(o\in\partial\cI_\lambda)}\right].
\end{eqnarray*}
On the other hand, recall from equation (\ref{eq:deg}) that
\begin{eqnarray*}
\EE\left[\left(\deg_{\cI_\lambda}(o)-2\right){\bf 1}_{(o\in\cI_\lambda)}\right] & = & -2\EE\left[\frac{{\bf 1}_{(o\in\cI_\lambda)}}{|\cc(\cI_\lambda,o)|}\right],
\end{eqnarray*}
and from the identities (\ref{eq:main}) and (\ref{eq:deg}) that
\begin{eqnarray*}
\label{degbis}
\PP\left(o\in\partial\cI_\lambda\right) & = & \EE\left[\frac{{\bf 1}_{(o\in\cI_\lambda)}}{|\cc(\cI_\lambda,o)|}\right]-\mu_\cL(\{\lambda\}).
\end{eqnarray*}
Combining those three identities, we arrive at
\begin{eqnarray*}
\EE\left[\left(\deg(o)-2\right){\bf 1}_{(o\in\cS)}\right] & = & \EE\left[\left(\deg(o)+\deg_{\cI_\lambda}(o)-2\right){\bf 1}_{(o\in\partial\cI_\lambda)}\right]-2\EE\left[\frac{{\bf 1}_{(o\in\cI_\lambda)}}{|\cc(\cI_\lambda,o)|}\right]\\
& \le & 2(\Delta-2)\EE\left[\frac{{\bf 1}_{(o\in\cI_\lambda)}}{|\cc(\cI_\lambda,o)|}\right]\\
& \le & \frac {2(\Delta-2)}{\tau(\lambda)}\PP\left(o\in\cS\right),
\end{eqnarray*}
which is exactly the claim (iii). 
\end{proof}
 \begin{lemma}\label{lm:anchored}Let $\cL$ be a unimodular network concentrated on trees with degrees in $\{2,\ldots,\Delta\}$. Suppose that it admits an $\varepsilon-$thin set in the sense of Lemma \ref{lm:thin} above. Then the event $$\left\{(G,o)\in\cGs\colon \ii^\star(G,o)\le \frac{3\Delta\varepsilon}{2}\right\}$$
has positive probability under $\cL$. 
\end{lemma}
\begin{proof}
We first work under the extra assumption that the $\varepsilon-$thin set $\cS$ achieves  equality in condition (ii), i.e., the subgraph induced by $\cS$ is a disjoint union of bi-infinite paths (called a \emph{line ensemble} in \cite{BSV}). On the event $\{o\in\cS\}$, let $\cS_n\subseteq \cS$ denote the subpath consisting of the $2n+1$ vertices of $\cS$ that lie  within distance $n$ from $o$. By the Mass Transport Principle, 
\begin{eqnarray*}
\EE\left[{\bf 1}_{(o\in\cS)}\sum_{x\in\cS_n}(\deg(x)-2)\right] & = & 
(2n+1)\EE\left[{\bf 1}_{(o\in\cS)}(\deg(o)-2)\right].
\end{eqnarray*}
Now, the sum on the left-hand side is exactly $|\partial \cS_n|-2$, while the expectation on the right-hand side is at most $\varepsilon\PP\left(o\in\cS\right)$, by  definition of $\varepsilon-$thinness. Rearranging, we arrive at
\begin{eqnarray*}
\EE\left[{\bf 1}_{(o\in\cS)}\frac{|\partial\cS_n|}{|\cS_n|}\right] & \le & 
\left(\varepsilon+\frac{2}{2n+1}\right)\PP\left(o\in\cS\right).
\end{eqnarray*}
On the other hand, on $\{o\in\cS\}$, we have by definition of the anchored isoperimetric constant,  
\begin{eqnarray*}
\ii^\star(G,o) & \le & \liminf_{n\to\infty}\frac{|\partial\cS_n|}{|\cS_n|}.
\end{eqnarray*}
By Fatou's Lemma, we conclude that
\begin{eqnarray*}
\EE\left[\ii^\star(G,o){\bf 1}_{(o\in\cS)}\right]  \le  \varepsilon\,\PP\left(o\in\cS\right),
\end{eqnarray*}
so that the event $\left\{(G,o)\in\cGs\colon \ii^\star(G,o)\le \varepsilon\right\}$ has positive probability under $\cL$. In the general case, we may always ``extract'' from $\cS$ a random subset $\cS'$ with internal degree exactly $2$ in a unimodular way, thanks to \cite[Proposition 5.4] {BSV} (we actually apply this Proposition to the unimodular law $\cL'$ obtained from $\cL$  by conditioning on $o\in \cS$ and restricting the  graph to $\cc(\cS,o)$, with marks on the vertices to keep track of their degrees in the original graph). The resulting set $\cS'\subseteq \cS$ satisfies
\begin{enumerate}[(i)]
\item $\PP\left(o\in\cS'\right)\ge \frac{2}{3\Delta}\PP\left(o\in\cS\right)$;
\item $\PP\left[\deg_{\cS'}(o)\,|\,o\in\cS'\right)=1$.
\end{enumerate}
We note that the construction of $\cS'$ may use some external randomization; the crucial point is that the resulting network (with $\cS'$ encoded as vertex marks) is unimodular,  so that the above use of the Mass Transport Principe remains valid. Note that $\cS'$ is $\varepsilon'-$thin with $\varepsilon'=\frac{3\Delta\varepsilon}{2}$ since by construction
\begin{eqnarray*}
\EE\left[(\deg(o)-2){\bf 1}_{o\in\cS'}\right] & \le & \EE\left[(\deg(o)-2){\bf 1}_{o\in\cS}\right]\\
& \le & \varepsilon\PP\left(o\in\cS\right)\\
& \le & \frac{3\Delta\varepsilon}{2}\PP\left(o\in\cS'\right).
\end{eqnarray*}
We may thus apply the first part of the proof with $(\cS',\varepsilon')$ instead of $(\cS,\varepsilon)$, and the claim follows. 
\end{proof}

 \bibliographystyle{abbrv}
\bibliography{biblio}

\begin{thebibliography}{10}

\bibitem{pointwise}
M.~Ab{\'e}rt, A.~Thom, and B.~Vir{\'a}g.
\newblock Benjamini-{S}chramm convergence and pointwise convergence of the
  spectral measure.
\newblock {\em In preparation.}, 2016.

\bibitem{uni}
D.~Aldous and R.~Lyons.
\newblock Processes on unimodular random networks.
\newblock {\em Electronic Journal of Probability}, 12:no. 54, 1454--1508, 2007.

\bibitem{obj}
D.~Aldous and J.~M. Steele.
\newblock The objective method: probabilistic combinatorial optimization and
  local weak convergence.
\newblock In {\em Probability on discrete structures}, volume 110 of {\em
  Encyclopaedia Math. Sci.}, pages 1--72. Springer, Berlin, 2004.

\bibitem{2015arXiv150507412B}
{\'A}.~{Backhausz} and B.~{Vir{\'a}g}.
\newblock {Spectral measures of factor of i.i.d. processes on vertex-transitive
  graphs}.
\newblock {\em ArXiv e-prints}, May 2015.

\bibitem{unitree}
I.~Benjamini, R.~Lyons, and O.~Schramm.
\newblock Unimodular random trees.
\newblock {\em Ergodic Theory Dynam. Systems}, 35(2):359--373, 2015.

\bibitem{rec}
I.~Benjamini and O.~Schramm.
\newblock Recurrence of distributional limits of finite planar graphs.
\newblock {\em Electronic Journal of Probability}, 6:no. 23, 13, 2001.

\bibitem{Bordenave2015}
C.~Bordenave.
\newblock On quantum percolation in finite regular graphs.
\newblock {\em Annales Henri Poincar{\'e}}, 16(11):2465--2497, 2015.

\bibitem{BorOptimization}
C.~Bordenave.
\newblock {\em Lecture notes on random graphs and probabilistic combinatorial
  optimization}.
\newblock 2016.
\newblock In preparation.

\bibitem{BorSpectrum}
C.~Bordenave.
\newblock {\em Spectrum of random graphs}.
\newblock 2016.
\newblock In preparation.

\bibitem{resolvent}
C.~Bordenave and M.~Lelarge.
\newblock Resolvent of large random graphs.
\newblock {\em Random Structures Algorithms}, 37(3):332--352, 2010.

\bibitem{rank}
C.~Bordenave, M.~Lelarge, and J.~Salez.
\newblock The rank of diluted random graphs.
\newblock {\em The Annals of Probability}, 39(3):1097--1121, 2011.

\bibitem{BSV}
C.~{Bordenave}, A.~{Sen}, and B.~{Vir{\'a}g}.
\newblock {Mean quantum percolation}.
\newblock {\em ArXiv e-prints}, Aug. 2013.

\bibitem{anchored}
D.~Chen and Y.~Peres.
\newblock Anchored expansion, percolation and speed.
\newblock {\em Ann. Probab.}, 32(4):2978--2995, 2004.
\newblock With an appendix by G{\'a}bor Pete.

\bibitem{cds}
D.~M. Cvetkovi{\'c}, M.~Doob, and H.~Sachs.
\newblock {\em Spectra of graphs}.
\newblock Johann Ambrosius Barth, Heidelberg, third edition, 1995.
\newblock Theory and applications.

\bibitem{sofic}
G.~Elek.
\newblock On the limit of large girth graph sequences.
\newblock {\em Combinatorica}, 30(5):553--563, 2010.

\bibitem{mckay}
B.~McKay.
\newblock The expected eigenvalue distribution of a large regular graph.
\newblock {\em Linear Algebra Appl.}, 40:203--216, 1981.

\bibitem{muller}
V.~M{\"u}ller.
\newblock On the spectrum of an infinite graph.
\newblock {\em Linear Algebra Appl.}, 93:187--189, 1987.

\bibitem{2016arXiv160902209R}
M.~{Rahman}.
\newblock {A lower bound on the spectrum of unimodular networks}.
\newblock {\em ArXiv e-prints}, Sept. 2016.

\bibitem{reedsimon}
M.~Reed and B.~Simon.
\newblock {\em Methods of modern mathematical physics. {I}. {F}unctional
  analysis}.
\newblock Academic Press, New York, 1972.

\bibitem{PhD}
J.~Salez.
\newblock {\em Some implications of local weak convergence for sparse random
  graphs}.
\newblock Theses, {Universit{\'e} Pierre et Marie Curie - Paris VI ; Ecole
  Normale Sup{\'e}rieure de Paris - ENS Paris}, July 2011.

\bibitem{eigenvalues}
J.~Salez.
\newblock Every totally real algebraic integer is a tree eigenvalue.
\newblock {\em Journal of Combinatorial Theory. Series B}, 111:249--256, 2015.

\end{thebibliography}
\end{document}